\documentclass[12pt]{article}
\usepackage{amsmath,amssymb,amsfonts,amsrefs,esint,mathtools,bm}
\usepackage[dvipdfmx]{graphicx}
\usepackage{authblk}
\mathtoolsset{showonlyrefs=true}
\numberwithin{equation}{section}

\usepackage{amsthm}

\theoremstyle{plain}
\newtheorem{dfn}{Definition}[section]
\newtheorem{thm}{Theorem}[section]
\newtheorem{thm*}{Theorem}
\newtheorem{prop}{Proposition}[section]
\newtheorem{lem}{Lemma}[section]

\theoremstyle{definition}
\newtheorem{rem}{Remark}[section]

\usepackage{appendix}

\newcommand{\R}{\mathbb{R}}
\newcommand{\N}{\mathbb{N}}
\allowdisplaybreaks

\title {Double phase flow under Lavrentiev phenomenon} 
\author{Yoshiki Kaiho\footnote{\texttt{kaiho.yoshiki.p7@dc.tohoku.ac.jp}}}
\affil{Tohoku University, Sendai, 980-8578, Japan}

\date{\today} 
\begin{document}
\maketitle

\begin{abstract}
This paper deals with parabolic equations associated with double phase functionals.
It is known that double phase functionals may exhibit the Lavrentiev phenomenon, which indicates an existence of a singular minimizer.
Our aim is to investigate the process of the associated double phase flow evolving toward the singular minimizer.
For this purpose, we study whether solutions can be approximated by smooth functions.
We first observe a phenomenon that we call finite-time loss of smooth approximability;
more precisely, we prove that the flow eventually ceases to be smoothly approximable.
We also establish quantitative estimates on the time of loss.
On the other hand, we investigate a phenomenon that we call short-time persistence of smooth approximability;
we prove that the smooth approximability persists for a short time 
provided that the initial datum is regular 
and that the functional is nondegenerate with respect to the gradient variable.
The results concerning finite-time loss of smooth approximability are derived from the evolution variational inequality,
whereas the short-time persistence result is obtained by applying analytic semigroup theory.
The novelty of this paper lies in studying the dynamical aspect of the Lavrentiev phenomenon,
which is usually regarded as a stationary phenomenon. 

\vspace{5pt}
\noindent \textbf{Key words:}\,parabolic double phase equation, Lavrentiev phenomenon \\
\textbf{AMS Classification:}\,35K61, 35D35, 35B65
\end{abstract}

\section{Introduction}\label{sec:intro}

\subsection{Background}

When one considers the minimization of variational functionals, the Lavrentiev phenomenon may occur.
Let
$$
\mathcal{F}(u) = \int_{\Omega} f(x,\nabla u)\,dx,
$$
where $n \in \N$, $\Omega \subset \R^n$ is a bounded domain and $f:\Omega \times \R^n \rightarrow \R$ is a Carathéodory function.
Then, clearly, we have
$$
\inf_{u:\mathrm{all}} \mathcal{F}(u) \le \inf_{u:\mathrm{regular}} \mathcal{F}(u).
$$
In some situations, however, the strict inequality 
$$
\inf_{u:\mathrm{all}} \mathcal{F}(u) < \inf_{u:\mathrm{regular}} \mathcal{F}(u)
$$
holds.
This is called the Lavrentiev phenomenon.
The first example of this phenomenon was given in Lavrentiev's work \cite{La}.
The occurrence of this phenomenon indicates low regularity of minimizers
in the sense that they cannot be approximated by smooth functions in the energy topology.
The purpose of this paper is to provide a dynamical interpretation of the Lavrentiev phenomenon 
by investigating how the associated parabolic flows evolves toward such a singular minimizer.



In this paper, we focus on the following functional, which may exhibit the Lavrentiev phenomenon:
\begin{align}\label{eq:double_phase_energy}
\varphi(u) 
\coloneq \int_{\Omega} \left( \left( |\nabla u|^2 + \mu \right)^{\frac{p}{2}} + a(x)\left(|\nabla u|^2 + \mu \right)^{\frac{q}{2}} \right)\,dx,
\end{align}
where $1 < p < q < \infty$, $a:\Omega \rightarrow [0,\infty)$ is a measurable function and $\mu \ge 0$.
The functional $\varphi$ is commonly referred to as a double phase functional and
is nondegenerate with respect to the gradient variable when $\mu > 0$.
Double phase functionals are often interpreted as models for the elastic energy of composite materials consisting of two phases: a soft phase and a hard phase.
Roughly speaking, for such functionals, if $\frac{q}{p}$ is sufficiently large, the Lavrentiev phenomenon may occur.
Indeed, the following result is known, which was obtained by Zhikov for $n=2$
and Esposito, Leonetti and Mingione for $n \ge 2$ (see \cite{EsLM,Zh}).
\begin{thm}[\cite{EsLM,Zh}]\label{prop:Lav}
Let $n \ge 2$, $\Omega$ be the unit ball of $\R^n$,
$\alpha \in (0,1]$ and $p < n < n + \alpha < q$.
Then there exist a boundary datum $g \in C^{\infty}(\overline{\Omega})$,
a coefficient function $a \in C^{0,\alpha}(\overline{\Omega})$ and a constant $\mu_0 > 0$ such that
if $\mu \in [0,\mu_0]$, then
$$
\inf_{u \in g + W^{1,p}_0(\Omega)} \varphi(u) < \inf_{u \in g + C_c^{\infty}(\Omega)} \varphi(u).
$$
\end{thm}
\begin{rem}
This result is usually stated in the degenerate case and for some Lipschitz boundary datum.
By a simple approximation argument, this result holds in the nondegenerate case and for some smooth boundary datum.
\end{rem}
\begin{rem}
In the proof of Theorem\,\ref{prop:Lav}, the condition $p < n$ is essential.
Recently, under the general condition $q > p + \alpha \max\{1,\frac{p-1}{n-1}\}$,
Balci, Diening and Surnachev in \cite[Theorem\,34]{BaDS} gave the new examples of the Lavrentiev phenomenon.
\end{rem}
Since the seminal work of Zhikov\,\cite{Zh},
two directions have emerged in the regularity theory of double phase functionals.
The first concerns the singular minimizers.
In this direction, together with \cite{BaDS,EsLM,Zh} already mentioned,
we refer the reader to \cite{BaDS_1,FoMM}.
The second concerns the regularity of the minimizers and stationary solutions under the sufficient condition excluding the Lavrentiev gap.
This includes higher integrability \cite{EsLM}, H\"{o}lder continuity for the gradient \cite{CoM1,CoM2,BaCM} and Calder\'{o}n--Zygmund estimates \cite{CoM3,DeFM0}.
More recently, the regularity results for the parabolic double phase equations have begun to develop.
For example, higher integrability \cite{Si,KiKM,KiS}, Lipschitz bound \cite{DeF}, Calder\'{o}n--Zygmund estimate \cite{Ki2} are available under assumptions analogous to those in the stationary case.
In contrast to the previous studies,
we are interested in the process of solutions to parabolic double phase equations evolving toward singular minimizers in the presence of the Lavrentiev phenomenon.

\subsection{Formulation}
In this paper, we consider the following parabolic equation driven by the double phase functional: 
\begin{equation}\label{eq:main}
\left\{
\begin{aligned}
\partial_t u = \mathrm{div}\left\{ p(|\nabla u|^2 + \mu)^{\frac{p-2}{2}}\nabla u + qa(x)(|\nabla u|^2 + \mu)^{\frac{q-2}{2}}\nabla u \right\}
\quad &\text{in}\,\Omega \times (0,\infty),\\
u = g \quad &\text{on}\, \partial\Omega \times (0,\infty),\\
u = u_0\quad &\text{in}\,\Omega,
\end{aligned}
\right.
\end{equation}
where $\Omega \subset \R^n\,(n \ge 2)$ is a bounded domain,
$1 < p < q < \infty$, $\mu \in [0,1]$ 
and $0 \le a \in C^{0,\alpha}(\overline{\Omega})\,(\alpha \in (0,1])$.
Here $C^{0,\alpha}(\overline{\Omega})$ is the space of $\alpha$-Holder continuous functions on $\overline{\Omega}$.
Moreover, $u_0$ is the initial datum and $g$ is the time-independent boundary datum.

Let us assume
\begin{align}\label{assmpt:boundary}
  g \in L^2(\Omega) \cap W^{1,D}(\Omega),
\end{align}
and
\begin{align}\label{eq:expo_cond}
\left[p < n
\quad
\text{and}
\quad
q \le \frac{(n+\alpha)p}{n-p}\right]
\quad
\text{or}
\quad
p \ge n.
\end{align}
Here $W^{1,D}(\Omega)$ is the generalized Orlicz space, consisting of $W^{1,p}$-functions with the finite double phase energy (see Subsection\,2.1, for the precise definitition).
Under these two assumptions, the general theory of Brezis \cite{Br_0} yields the global well-posedeness of \eqref{eq:main} in the class of $L^2$-solutions;\,see Subsection\,2.1.
\begin{rem}
One can choose $p,q,\alpha$ so that $p < n < n + \alpha < q$ and $q \le \frac{(n + \alpha)p}{n-p}$. 
Hence, there are parameter regimes in which the Lavrentiev phenomenon occurs while the global well-posedness holds;\,see Theorem\,\ref{prop:Lav}.
\end{rem}

Let us denote by $S(t)u_0$ the unique $L^2$-solution to \eqref{eq:main} with initial datum $u_0$ (see also Remark\,\ref{rem:contraction}),
and let
$$
X \coloneq g+\overline{C_c^\infty(\Omega)}^{W^{1,D}(\Omega)}
$$
be the affine space obtained by translating the closure of $C_c^\infty(\Omega)$ in $W^{1,D}(\Omega)$ by $g$.
Suppose the Lavrentiev phenomenon occurs. 
Then the minimizer in the larger admissible class does not belong to $X$.
On the other hand, for any suitable initial datum $u_0$, 
the $L^2$-solution $S(t)u_0$ evolves toward the singular minimizer
because of the strict convexity of the double phase functional.
Here we investigate the evolution of the flow $S(t)u_0$,
starting from a regular initial datum $u_0$, 
from the following two perspectives:
\begin{itemize}
  \item whether $S(t)u_0$ eventually ceases to belong to $X$;
  \item whether $S(t)u_0$ remains in $X$ for some initial time interval.
\end{itemize}
We refer to the former phenomenon as finite-time loss of smooth approximability 
and to the latter as short-time persistence of smooth approximability. 

The novelty of this paper is to capture the dynamical aspect of the Lavrentiev phenomenon,
which is commonly seen as a stationary problem, 
from the viewpoints of finite-time loss of smooth approximability and short-time persistence of smooth approximability.

\subsection{Main result}
Let $\varphi$ be the double phase functional defined in \eqref{eq:double_phase_energy}.
We first state results about the finite-time loss of smooth approximability.
Assume that
\begin{align}\label{assmpt:2.2.1}
  p \ge \frac{2n}{n+2},
\end{align}
and
\begin{align}\label{assmpt:Lav}
\inf_{v \in g + W^{1,D}_0(\Omega)} \varphi(v) < \inf_{v \in X} \varphi(v).
\end{align}
The assumption \eqref{assmpt:2.2.1} is used to obtain the embedding $W^{1,D}(\Omega) \subset L^2(\Omega)$;
it can be replaced by the boundedness of $g$, as explained in Remark\,\ref{rem:boundedness}.
And the assumption \eqref{assmpt:Lav} is precisely the occurrence of the Lavrentiev phenomenon.
Moreover, let $u_{min} \in g + W^{1,D}_0(\Omega)$ and $u_* \in X$ be such that
$$
\varphi(u_{min}) = \inf_{v \in g + W^{1,D}_0(\Omega)} \varphi(v)
\quad
\text{and}
\quad
\varphi(u_*) = \inf_{v \in X} \varphi(v).
$$
The existence and uniqueness of $u_{min}$ and $u_*$ follow from the direct method and the strict convexity of $\varphi$.

Under these assumptions, we obtain the following result.
\begin{thm}\label{thm:fti}
Assume \eqref{assmpt:boundary}, \eqref{eq:expo_cond}, \eqref{assmpt:2.2.1} and \eqref{assmpt:Lav}.
\begin{itemize}
\item[\rm{(i)}] For any $u_0 \in X$, it holds that
\begin{align}
      S(t)u_0 \not \in X
      \quad
      \text{if}
      \quad
      t > \frac{\|u_{min} - u_0\|^2_{L^2(\Omega)}}{2\delta},
\end{align}
where $\delta \coloneq \varphi(u_*) - \varphi(u_{min})(>0)$.
  \item[\rm{(ii)}] For any $\varepsilon > 0$, there exists $\eta > 0$ depending on $\varepsilon$ such that
if $u_0 \in X$ satisfies $\|u_* - u_0\|_{L^2(\Omega)} < \eta$, then it holds that
$$
S(t)u_0 \not \in X
\quad
\text{if}
\quad
t > \varepsilon.
$$
In particular, it follows that
$$
S(t)u_* \not \in X
\quad
\text{if}
\quad
t > 0.
$$
\end{itemize}
\end{thm}
\begin{rem}
Theorem\,\ref{thm:fti}\,(i) asserts that,
for any initial datum in $X$,
the flow $S(t)u_0$ eventually ceases to $X$. 
Moreover, it shows that the upper bound on the time of loss decreases as $u_0$ approaches $u_{min}$ in $L^2(\Omega)$.
The assertion \rm{(ii)} further shows that $u_* \in X$ plays a role analogous to that of $u_{min}$: 
if the initial datum is sufficiently close to $u_*$ in $L^2(\Omega)$, 
the loss of smooth approximability occurs arbitrarily early.
\end{rem}

Next we turn our attention to short-time persistence of smooth approximation.
Let us suppose the following condition:
\begin{align}\label{assmpt:3.1}
\mu > 0,
\quad
\gamma \in (0,1),
\quad
\partial\Omega \in C^{3,\gamma},
\quad
g \in C^{1,\gamma}(\overline{\Omega}).
\end{align}
The assumption $\mu > 0$ ensures that the operator in \eqref{eq:main} is nondegenerate and sufficiently smooth with respect to the gradient variable.
Then we obtain the following theorem.
\begin{thm}\label{thm:ftpr}
Assume \eqref{eq:expo_cond} and \eqref{assmpt:3.1}.
Let $u_0 \in C^{1,\gamma}(\overline{\Omega})$ with $u|_{\partial \Omega} = g$ and $S(t)u_0$ is the unique $L^2$-soluiton of \eqref{eq:main}.
Then there exists $T(u_0) > 0$, depending on $\mu$ in particular, such that
$$
S(t)u_0 \in X 
\quad
\text{if}
\quad
0 \le  t  \le T(u_0).
$$
\end{thm}

\begin{rem}
It is an interesting problem whether the same result holds as Theorem\,\ref{thm:ftpr} in the degenerate case. 
We also remark that the degenerate case cannot be obtained simply 
by passing to the limit ($\mu \to 0$) because the existence time depends on $\mu$ 
in such a way that $T(u_0) \rightarrow 0$ as $\mu \rightarrow 0$.
\end{rem}

\subsection{Methods and organization}
Let us explain the strategy for proving Theorem\,\ref{thm:fti} and \ref{thm:ftpr}. 
For Theorem\,\ref{thm:fti}, we use the energy decreasing property of the $L^2$-solution.
By this property, under the Lavrentiev phenomenon \eqref{assmpt:Lav}, 
we can find a sufficiently large time $T$ 
such that
$\varphi(u_{min}) \le \varphi(S(t)u_0) < \varphi(u_*)$ for any $t \ge T$.
And this implies $\varphi(S(t)u_0) \not\in X$ for any $t \ge T$.
To obtain a quantitative estimate for such a time $T$, 
we further utilize the evolution variational inequality \eqref{eq:evi};
see also \cite{AmGS} for a general overview.

In order to prove Theorem\,\ref{thm:ftpr},
we first construct a sufficiently regular local-in-time solution $u(t)$,
which satisfies $u(t) \in X$ in the existence time.
And we identify it with the $L^2$-solution.
For the construction of such a regular solution,
we restrict ourselves in the nondegenerate case\,($\mu > 0$).
First note that the standard existence result for quasilinear parabolic equations,
such as those in \cite[Chapter\,V--VI]{La},
cannot be applied directly even if $\mu > 0$.
This is because of non-standard growth which comes from the degeneracy of the coefficient function $a$.
One might try to remove this degeneracy by replacing $a$ with $a(x) + \varepsilon$.
However, this approximation does not seem to be effective in the present setting. 
Indeed, the passage to the limit would require suitable a priori estimates independent of $\varepsilon$,
and the existence of such estimates would be incompatible with the occurrence of the Lavrentiev phenomenon.
For this reason, we adopt the analytic semigroup approach, which is flexible enough to handle this type of growth (see \cite{Am1,Lu0,MaW,PrS} and the references therein).
Combining the analytic semigroup generation result in the space $C^{-1,\beta}(\overline{\Omega})$\,\cite{Ve} with the local existence result for abstract fully nonlinear equations\,\cite{Lu0}, 
we can construct such a local-in-time smooth solution.
We identify this solution with the $L^2$-solution by uniqueness.

Finally we state the plan of this paper.
In Section\,2, we show the global well-posedness of \eqref{eq:main} in the framework of $L^2$-solutions
and then we prove Theorem\,\ref{thm:fti}.
In Section\,3, we construct a relatively regular local-in-time solution
and identify it with the $L^2$-solution.

\section{Finite-time loss of smooth approximability}\label{sec:2}
\subsection{Global well-posedness of $L^2$-solution}
In this subsection, we prove the global well-posedness of \eqref{eq:main} in the class of $L^2$-solutions,
based on the general theory developed in \cite{Br_0}.
For this purpose we adapt the argument of \cite{AkM}, where the $p(x)$-laplacian is treated.

We first recall the definitions and properties of generalized Orlicz spaces associated with double phase functionals.
For more details in this topic, we refer to the book \cite{HaH}.
Let us define the double phase modular $\rho$ by
$$
\rho(u) 
\coloneq \int_{\Omega} \left( |u|^p + a(x)|u|^{q} \right)\,dx
$$
for any measurable function $u$.
Using the modular $\rho$, we define the associated generalized Orlicz space by
$$
L^{D}(\Omega) \coloneqq \left\{u:\Omega \rightarrow \R: measurable : \rho(u) < \infty \right\}
$$
endowed with the Luxemburg norm
$$
\|u\|_{L^D(\Omega)} = \inf\left\{\lambda > 0: \rho \left( \frac{u}{\lambda} \right) \le 1\right\}.
$$
Recall the following relationship between $L^{D}$-norm and the modular $\rho$ (\cite[Lemma\,3.2.9]{HaH}).
\begin{prop}\label{prop:eqiv_norm_modular}
It holds that 
$$
  \min\left\{\|u\|_{L^{D}(\Omega)}^{p},\|u\|_{L^{D}(\Omega)}^{q}\right\}
  \le \rho(u)
  \le \max\left\{\|u\|_{L^{D}(\Omega)}^{p},\|u\|_{L^{D}(\Omega)}^{q}\right\}
$$
for any $u \in L^{D}(\Omega)$.
\end{prop}

We next define the Sobolev space $W^{1,D}(\Omega)$ as follows:
$$
W^{1,D}(\Omega) \coloneqq \left\{u \in L^{D}(\Omega) \cap W^{1,p}(\Omega): 
\nabla u \coloneq \left(D_i u \right)_{1 \le i \le n} \in (L^{D}(\Omega))^n \right\}
$$
with the norm
$$
\|u\|_{W^{1,D}(\Omega)} = \|u\|_{L^{D}(\Omega)} + \|\nabla u\|_{L^{D}(\Omega)},
$$
where $\|\nabla u\|_{L^{D}}$ denotes the $L^{D}(\Omega)$-norm of $|\nabla u|$.

Thanks to \cite[Theorem\,6.1.4]{HaH}, the following proposition holds.
\begin{prop}\label{prop:reflexive}
If $a \in L^{\infty}(\Omega)$, then $L^{D}(\Omega)$ and $W^{1,D}(\Omega)$ are uniformly convex. Hence they are reflexive.
\end{prop}
Furthermore, we set
$$
W^{1,D}_0(\Omega) \coloneq W^{1,p}_0(\Omega) \cap W^{1,D}(\Omega).
$$
\begin{rem}\label{rem:inclusion}
Note that
$\overline{C_c^{\infty}(\Omega)}^{W^{1,D}} = \overline{W^{1,q}_0(\Omega)}^{W^{1,D}} \subset W^{1,D}_0(\Omega)$.
In particular, the last inclusion may be proper when the Lavrenteiv phenomenon occurs.
\end{rem}

We use a Poincar\'{e}-type inequality to establish the lower semicontinuity of the double phase functional (see Lemma\,\ref{lem:lsc}).
\begin{prop}\label{prop:poincare}
Assume \eqref{eq:expo_cond}.
Then there exists a constant $C > 0$ depending only on $n,p,q,\alpha,\|a\|_{C^{0,\alpha}(\overline{\Omega})}$ and $\Omega$ such that
$$
\|a^{\frac{1}{q}}u\|_{L^{q}(\Omega)} \le C \left( 
\|a^{\frac{1}{q}}\nabla u\|_{L^{q}(\Omega)} + \|\nabla u\|_{L^p(\Omega)} \right)
$$
for any $u \in W^{1,D}_0(\Omega)$.
\end{prop}
\begin{proof}
Take any $u \in W^{1,D}_0(\Omega)$.
Then we have $u \in W^{1,p}_0(\Omega)$ and $a^{\frac{1}{q}}|\nabla u| \in L^q(\Omega)$.
We first consider the case $p \ge n$ or $\left[p < n\,\text{and}\,\frac{np}{n-p} \ge q\right]$.
The Sobolev embedding theorem implies
\begin{align}
  \|a^{\frac{1}{q}}u\|_{L^{q}(\Omega)}
  \le \|a\|_{L^{\infty}(\Omega)}^{\frac{1}{q}}\|u\|_{L^q(\Omega)}
  \le C\|\nabla u\|_{L^{p}(\Omega)},
\end{align}
where $C$ depends only on $n,p,q,\|a\|_{L^{\infty}(\Omega)}$ and $\Omega$.

It remains to consider the case $p<n$ and $\frac{(n + \alpha)p}{n-p} \ge q > \frac{np}{n-p}$.
We can find $\beta \in (0,\alpha]$ such that $\frac{(n + \beta)p}{n-p} = q$.
Then \cite[Lemma\,7.14]{GT} yields, for a.e.\,$x \in \Omega$, 
\begin{align}
a^{\frac{1}{q}}(x)|u(x)|
&\le \int_{\Omega}\frac{a^{\frac{1}{q}}(x)|\nabla u|(y)}{|x-y|^{n-1}}\,dy\\
&\le \int_{\Omega}\frac{a^{\frac{1}{q}}(y)|\nabla u|(y)}{|x-y|^{n-1}}\,dy
+ [a]_{C^{0,\alpha}(\overline{\Omega})}^{\frac{1}{q}}\int_{\Omega} 
\frac{|x-y|^{\frac{\alpha}{q}}|\nabla u|(y)}{|x-y|^{n-1}}\,dy\\
&= I_1(a^{\frac{1}{q}}|\nabla u|)(x)
+ [a]_{C^{0,\alpha}(\overline{\Omega})}^{\frac{1}{q}}
\int_{\Omega} 
\frac{|x-y|^{\frac{\alpha}{q} - \frac{\beta}{q}}|\nabla u|(y)}{|x-y|^{n-\left(1+\frac{\beta}{q} \right)}}\,dy \\
&\le I_1(a^{\frac{1}{q}}|\nabla u|)(x) + 
[a]_{C^{0,\alpha}(\overline{\Omega})}^{\frac{1}{q}}|\mathrm{diam}\,\Omega|^{\frac{\alpha}{q} - \frac{\beta}{q}}I_{1+\frac{\beta}{q}}(|\nabla u|)(x),
\end{align}
where $I_{\gamma}(\cdot)\,(\gamma \in (0,n])$ is the Riesz potential.
Note that $q = \frac{np}{n-\left(1+ \frac{\beta}{q}\right)p}$ and
the standard mapping properties of the Riesz potential (see\,\cite[LEMMA\,1.34\,and\,THEOREM\,1.36]{KiLV},\,for example) yield
\begin{align}
  \|a^{\frac{1}{q}}u\|_{L^{q}(\Omega)}
  &\le \|I_1(a^{\frac{1}{q}}|\nabla u|)\|_{L^q(\Omega)} 
  + [a]_{C^{0,\alpha}(\overline{\Omega})}^{\frac{1}{q}}|\mathrm{diam}\,\Omega|^{\frac{\alpha}{q} - \frac{\beta}{q}}\|I_{1+\frac{\beta}{q}}(|\nabla u|)\|_{L^q(\Omega)} \\
  &\le C \left( \|a^{\frac{1}{q}}\nabla u\|_{L^q(\Omega)}
  + \|\nabla u\|_{L^p(\Omega)} \right),
\end{align}
where $C$ depends only on $n,p,q,\alpha,[a]_{C^{0,\alpha}(\overline{\Omega})}$ and $\Omega$.
The proof is completed.
\end{proof}

In the following, we use the shorthand notation
$$
W^{1,D}_g(\Omega) = g + W^{1,D}_0(\Omega).
$$

Fix $\mu \in [0,1]$
and we define the functional $\widetilde{\varphi}: L^{2}(\Omega) \rightarrow [0,\infty]$ by
\begin{equation}\label{eq:phi_tilde}
\widetilde{\varphi}(u)\coloneq
\left\{
\begin{aligned}
\varphi(u)
\quad
&\text{if}\,u \in L^{2}(\Omega) \cap W^{1,D}_g(\Omega),\\
\infty 
\quad
&\text{otherwise,}
\end{aligned}
\right.
\end{equation}
with $D(\widetilde{\varphi}) \coloneq L^{2}(\Omega) \cap W^{1,D}_g(\Omega)$.
Here $\varphi$ is the double phase functional defined in \eqref{eq:double_phase_energy}.
The subdiiferential operator $\partial{\widetilde{\varphi}}:L^2(\Omega) \rightarrow 2^{L^2(\Omega)}$ is given by
$$
\partial{\widetilde{\varphi}}(u)
\coloneq
\{w \in L^2(\Omega) : \widetilde{\varphi}(v) - \widetilde{\varphi}(u) \ge (w,v-u)_{L^2(\Omega)}\,\text{for any}\,v \in L^2(\Omega)\}
$$
for $u \in D(\widetilde{\varphi})$.
Here, $(\cdot,\cdot)_{L^2(\Omega)}$ denotes the $L^2(\Omega)$-inner product.
The domain of the subdifferential is defined by 
$$
D(\partial{\widetilde{\varphi}}) \coloneq \{u \in D(\widetilde{\varphi}) : \partial{\widetilde{\varphi}}(u) \neq \emptyset\}.
$$
We rewrite \eqref{eq:main} as the following abstract evolution equation in $L^2(\Omega)$:
\begin{equation}\label{eq:main_subdiff}
\left\{
\begin{aligned}
\frac{du}{dt}(t) + \partial{\widetilde{\varphi}}(u(t)) \ni 0
\quad 
&\text{in}\,L^2(\Omega)
\quad
\text{for}
\quad
t > 0,\\
u(0) = u_0.
\end{aligned}
\right.
\end{equation}
Now we state the definition of the $L^2$-solution to \eqref{eq:main_subdiff} (cf.\,\cite[DEFINITION\,3.1]{Br_0}).
\begin{dfn}\label{dfn:2.1}
For each $u_0 \in L^{2}(\Omega)$, a function $u \in C([0,\infty);L^2(\Omega))$ is a $L^2$-solution to \eqref{eq:main_subdiff} if the following conditions are satisfied:
\begin{itemize}
  \item $u \in W^{1,2}_{loc}((0,\infty);L^2(\Omega))$.
  \item For a.e.\,$t>0$, $u(t) \in D(\partial \widetilde{\varphi})$ and $u$ satisfies $\frac{du}{dt}(t) + \partial{\widetilde{\varphi}}(u(t)) \ni 0$.
  \item $u(0) = u_0$.
\end{itemize}
\end{dfn}


Now let us prove the global-wellposedness of \eqref{eq:main_subdiff} under Definition\,\ref{dfn:2.1}.
Here we consider the derivative of the double phase functional $\varphi$.
Observe that $\varphi$ is G$\hat{\text{a}}$teaux differentiable in $W^{1,D}(\Omega)$;
we have
\begin{align}
d\varphi(u)(v)
&\coloneqq
\lim_{t \rightarrow 0} \frac{\varphi(u + tv) - \varphi(u)}{t}\\
&= \int_{\Omega} \left( p(|\nabla u|^2 + \mu)^{\frac{p-2}{2}} + qa(x)(|\nabla u|^2 + \mu)^{\frac{q-2}{2}} \right)\nabla u \cdot \nabla v\,dx
\end{align}
for any $u,v \in W^{1,D}(\Omega)$.
Thanks to Proposition\,\ref{prop:eqiv_norm_modular},
$d\varphi$ belongs to $C(W^{1,D}(\Omega), (W^{1,D}(\Omega))^*)$,
where $(W^{1,D}(\Omega))^*$ is the dual space of $W^{1,D}(\Omega)$. 
Therefore the following lemma holds.
\begin{lem}\label{lem:diffrentiability}
The double phase functional $\varphi$ belongs to $C^1(W^{1,D}(\Omega),\R)$.
Moreover, it holds that 
$$
 d\varphi(u)(v) 
 = \int_{\Omega} \left( p(|\nabla u|^2 + \mu)^{\frac{p-2}{2}} + qa(x)(|\nabla u|^2 + \mu)^{\frac{q-2}{2}} \right)\nabla u \cdot \nabla v\,dx
$$
for any $u,\,v \in W^{1,D}(\Omega)$.
\end{lem}

We state the characterization of $D(\partial \widetilde{\varphi})$.
This characterization implies that the $L^2$-solution is also a weak solution (see\,\cite{KiKM,KiS}, for example).
The following lemma follows from Lemma\,\ref{lem:diffrentiability} 
and a standard argument,
so we omit the proof.
\begin{lem}\label{lem:charc_subdiff_domain}
Assume \eqref{assmpt:boundary} and \eqref{eq:expo_cond}.
Let $u \in D(\widetilde{\varphi})$, then the following assertions are equivalent:
\begin{itemize}
  \item[\rm{(i)}] $u \in D \left( \partial{\widetilde{\varphi}} \right)$.
  \item[\rm{(ii)}] There exists a unique function $w \in L^2(\Omega)$ such that
  \begin{align}\label{eq:Delta_DP}
  \int_{\Omega} \left( p(|\nabla u|^2 + \mu)^{\frac{p-2}{2}} + qa(x)(|\nabla u|^2 + \mu)^{\frac{q-2}{2}} \right)\nabla u \cdot \nabla \xi\,dx
  =
  \int_{\Omega} w\xi
  \end{align}
  for any $\xi \in L^2(\Omega) \cap W^{1,D}_0(\Omega)$.
\end{itemize}
Moreover, if $u \in D \left( \partial{\widetilde{\varphi}} \right)$, then $\partial{\widetilde{\varphi}}(u) = \{w\}$.
\end{lem}

Next we show the lower semicontinuity of the functional $\widetilde{\varphi}$ in $L^2(\Omega)$.
\begin{lem}\label{lem:lsc}
Assume \eqref{assmpt:boundary} and \eqref{eq:expo_cond}.
Then, $\widetilde{\varphi}$ is proper, convex and lower semicontinuous in $L^2(\Omega)$.
\end{lem}
\begin{proof}
Since $g \in D(\widetilde{\varphi})$, then $\widetilde{\varphi}$ is proper.
Moreover, it is obvious that $\widetilde{\varphi}$ is convex.
Finally, we show the lower semicontinuity of $\widetilde{\varphi}$ in $L^2(\Omega)$.
Let $\lambda \in \R$ be fixed and set
$$
[\widetilde{\varphi} \le \lambda] \coloneq \{u \in L^2(\Omega): \widetilde{\varphi}(u) \le \lambda\}.
$$
It suffices to prove that $[\widetilde{\varphi} \le \lambda]$ is closed in $L^2(\Omega)$.
Take $\{u_k\}_{k \in \N}$ and $u$ such that $u_k \rightarrow u$ in $L^2(\Omega)$.
In the following, we prove $u \in [\widetilde{\varphi} \le \lambda]$.

We first check $u \in D(\widetilde{\varphi})$.
Note that $\rho(|\nabla u_k|) \le \widetilde{\varphi}(u_k) \le \lambda$ for any $k \in \N$.
Moreover, by Proposition\,\ref{prop:poincare}, it follows that
\begin{align}
\rho(u_k) &= \int_{\Omega} \left(|u_k|^p + a(x)|u_k|^q \right)\,dx\\
&\le c\int_{\Omega} \left(|u_k-g|^p + a(x)|u_k-g|^q \right)\,dx
+ c\int_{\Omega} \left(|g|^p + a(x)|g|^q \right)\,dx \\
&\le c\int_{\Omega} |\nabla u_k - \nabla g|^p\,dx
+ c \left( \int_{\Omega} |\nabla u_k - \nabla g|^p\,dx \right)^{\frac{q}{p}}
+  c\int_{\Omega} a(x)|\nabla u_k - \nabla g|^q \,dx \\
&+ c\int_{\Omega} \left(|g|^p + a(x)|g|^q \right)\,dx \\
&\le c\int_{\Omega} |\nabla u_k|^p\,dx
+ c\left( \int_{\Omega} |\nabla u_k|^p\,dx \right)^{\frac{q}{p}}
+ c\int_{\Omega} a(x)|\nabla u_k|^q \,dx \\
&+ c\int_{\Omega} |\nabla g|^p\,dx
+ c\left( \int_{\Omega} |\nabla g|^p\,dx \right)^{\frac{q}{p}}
+ c\int_{\Omega} a(x)|\nabla g|^q \,dx
+ c\int_{\Omega} \left(|g|^p + a(x)|g|^q \right)\,dx \\
&\le c\max\{1, \rho(u_k)^\frac{q}{p}, \rho(|\nabla g|)^\frac{q}{p},\rho(g)\}
\le c\max\{1, \lambda^\frac{q}{p}, \rho(|\nabla g|)^\frac{q}{p},\rho(g)\},
\end{align}
where $c$ depends on $n,p,q,\alpha,\|a\|_{C^{0,\alpha}(\overline{\Omega})}$ and $\Omega$.
Then $\{u_k\}$ is a bounded sequence in $W^{1,D}_g(\Omega)$ by Proposition\,\ref{prop:eqiv_norm_modular}.
Since $W^{1,D}(\Omega)$ is reflexive by Proposition\,\ref{prop:reflexive},
we may extract a subsequence of $\{u_k\}$, still denoted by $\{u_k\}$, such that
$u_k \rightharpoonup u$ in $W^{1,D}(\Omega)$ weakly.
Moreover, we have $u \in W^{1,D}_g(\Omega)$ since $W^{1,D}_g(\Omega)$ is a closed convex set in $W^{1,D}(\Omega)$.
Thus it follows that $u \in D(\widetilde{\varphi})$.

It remains to show $\widetilde{\varphi}(u) \le \lambda$.
By Lemma\,\ref{lem:diffrentiability} and \cite[Corollary\,3.9]{Br}, 
$\varphi$ is weakly lower semicontinuous in $W^{1,D}(\Omega)$.
Therefore, we have
$
\widetilde{\varphi}(u) 
= \varphi(u)
\le \liminf_{k \rightarrow \infty} \varphi(u_k)
= \liminf_{k \rightarrow \infty} \widetilde{\varphi}(u_k)
\le \lambda.
$
The proof is completed.
\end{proof}

Note that $\overline{D(\partial \widetilde{\varphi})}^{L^2(\Omega)} = \overline{D( \widetilde{\varphi})}^{L^2(\Omega)}$ by Lemma\,\ref{lem:lsc} and \cite[PROPOSITION\,2.11]{Br_0}.
Moreover, observe that $\overline{D( \widetilde{\varphi})}^{L^2(\Omega)} = L^2(\Omega)$.
Therefore, Lemma\,\ref{lem:lsc} and \cite[THEOREM\,3.4 and 3.6]{Br_0} yield the following theorem:
\begin{thm}\label{thm:glw}
Assume \eqref{assmpt:boundary} and \eqref{eq:expo_cond}.
For each $u_0 \in L^2(\Omega)$, there exists a unique $L^2$-solution $u \in C([0,\infty);L^2(\Omega))$ to \eqref{eq:main_subdiff} such that $t \mapsto \widetilde{\varphi}(u(t))$ is absolutely continuous on $(0,\infty)$ and 
\begin{align}\label{eq:energy_decreasing}
\left\|\frac{du}{dt}\right\|^2_{L^2(\Omega)} + \frac{d}{dt}\widetilde{\varphi}(u(t)) = 0
\quad
\text{for}
\quad
\text{a.e.\,}
t \in (0,\infty).
\end{align}
In particular, if $u_0 \in D(\widetilde{\varphi})$, then $u \in W^{1,2}_{loc}([0,\infty);L^2(\Omega))$ and $t \mapsto \widetilde{\varphi}(u(t))$ is absolutely continuous on $[0,\infty)$.
Moreover, if $u_i$ is the unique $L^2$-solution of \eqref{eq:main_subdiff} with $u_{0,i} \in L^2(\Omega)$ for $i = 1,2$,
then it holds that
\begin{align}\label{eq:conti_dependence}
\|u_{1}(t) - u_{2}(t)\|_{L^2(\Omega)}
\le \|u_{0,1} - u_{0,2}\|_{L^2(\Omega)}
\quad
\text{for\,any}
\quad
t \in [0,\infty).
\end{align}
\end{thm}

\begin{rem}\label{rem:contraction}
\begin{enumerate}
  \item[(i)]\,In the following, for each $u_0 \in L^2(\Omega)$, 
  we denote by $S(t)u_0\,(t \ge 0)$ the unique $L^2$-solution with $u_0$. 
  Note that $\{S(t)\}_{t \ge 0}$ is a contraction semigroup on $L^2(\Omega)$ by the uniquness of $L^2$-solution and \eqref{eq:conti_dependence}.
  \item[(ii)]\,From \eqref{eq:energy_decreasing}, it follows that $t \mapsto \widetilde{\varphi}(S(t)u_0)$ is decreasing on $[0,\infty)$ for any $u_0 \in D(\widetilde{\varphi})$.
  \item[(iii)]\,The conclusion of Theorem\,\ref{thm:glw} remains valid with $(0,\infty)$ replaced by $(0,T)$ for any $T > 0$.
\end{enumerate}
\end{rem}

\begin{rem}
Since \eqref{eq:main_subdiff} has no forcing term, 
we can obtain better regularity for the solution (see \cite[THEOREM\,3.2]{Br_0}).
However, for the following argument we do not need such the properties.
Then we omit them for brevity.
\end{rem}



\subsection{Proof of Theorem\,\ref{thm:fti}}
We now demonstrate Theorem\,\ref{thm:fti}.
Recall the functionals $\varphi$ and $\widetilde{\varphi}$ defined in \eqref{eq:double_phase_energy} and \eqref{eq:phi_tilde}.
Note that under \eqref{assmpt:2.2.1} we have $D(\varphi) \coloneq D(\widetilde{\varphi}) = W^{1,D}_g(\Omega)$ and $\varphi = \widetilde{\varphi}$ in $D(\varphi)$ by the Sobolev embedding $W^{1,D}_0(\Omega) \subset W^{1,p}_0(\Omega) \subset L^2(\Omega)$.
Using the evolution variational inequality (see \eqref{eq:evi}),
we prove Theorem\,\ref{thm:fti}\,(i).


\begin{proof}[Proof\,of\,Theorem\,\ref{thm:fti}(i)]
Take any $v \in D(\varphi)$.
By Definition\,\ref{dfn:2.1}, 
we have $\frac{d}{dt}\left(S(\cdot)v\right)(t) + \partial{\varphi}(S(t)v) \ni 0$ for a.e.\,$t >0$.
Let $t > 0$ satisfy $\frac{d}{dt}\left(S(\cdot)v\right)(t) + \partial{\varphi}(S(t)v) \ni 0$.
From the definition of the subdifferential, it follows that
$$
\varphi(w) - \varphi(S(t)v) \ge \left(-\frac{d}{dt}\left(S(\cdot)v\right)(t), w - S(t)v\right)_{L^2(\Omega)}
\quad
\text{for any}
\quad
w \in D(\varphi).
$$
Here note that $\left(-\frac{d}{dt}\left(S(\cdot)v\right)(t), w - S(t)v\right)_{L^2(\Omega)} = \frac{1}{2}\frac{d}{dt}\left(\|S(\cdot)v - w\|^2_{L^2} \right)(t)$.
Then we have the following evolution variational inequality:
\begin{align}\label{eq:evi}
 \varphi\left(S(t)v\right) - \varphi(w) \le -\frac{1}{2}\frac{d}{dt}\left(\|S(\cdot)v - w\|^2_{L^2} \right)(t) 
\end{align}
for a.e.\,$t > 0$ and any $v,\,w \in D(\varphi)$.
Take any $\tau > 0$ and integrate \eqref{eq:evi} on $[0,\tau]$, we obtain
\begin{align}\label{eq:evi_2}
\int_{0}^{\tau} \left(\varphi\left(S(t)v\right) - \varphi(w)\right)\,dt
&\le -\frac{1}{2}\left( \|S(\tau)v - w\|^2_{L^2} -  \|v - w\|^2_{L^2}\right) \notag \\ 
&\le \frac{1}{2}\|v - w\|^2_{L^2}. 
\end{align}
Moreover, since $t \mapsto \varphi\left(S(t)v\right)$ is decreasing on $[0,\infty)$ (see Remark\,\ref{rem:contraction}\,(ii)), 
it follows that
$$
\int_{0}^{\tau} \left(\varphi\left(S(t)v\right) - \varphi(w)\right)\,dt
\ge \int_{0}^{\tau} \left(\varphi\left(S(\tau)v\right) - \varphi(w)\right)\,dt
= \tau \left(\varphi\left(S(\tau)v\right) - \varphi(w) \right).
$$
Therefore, we have
\begin{equation}\label{eq:evi_int}
 \varphi\left(S(\tau)v\right) - \varphi(w)  \le \frac{1}{2\tau}\|v - w\|^2_{L^2}
\end{equation}
for any $v, w \in D(\varphi)$. 
Here choose $v = u_0$ and $w = u_{min}$, then we obtain
$$
\varphi\left(S(\tau)u_0\right) - \varphi(u_{min}) \le \frac{1}{2\tau}\|u_0 - u_{min}\|^2_{L^2}
$$
and then
$$
\varphi\left(S(\tau)u_0\right) - \varphi(u_*) \le \frac{1}{2\tau}\|u_0 - u_{min}\|^2_{L^2} - \left( \varphi(u_*) - \varphi(u_{min}) \right).
$$
Here we put $\delta = \varphi(u_*) - \varphi(u_{min})\,(>0)$. 
Observe that $\frac{1}{2\tau}\|u_0 - u_{min}\|^2_{L^2} - \delta < 0$ implies $S(\tau)u_0 \not \in X$. 
Thus, if $\tau > \frac{\|u_0 -u_{min}\|^2}{2\delta}$, then $S(\tau)u_0 \not \in X$.
The proof is completed.
\end{proof}

From the convexity of $\varphi$, we obtain the following lemma:
\begin{lem}
Assume \eqref{assmpt:boundary}, \eqref{eq:expo_cond} and \eqref{assmpt:2.2.1}.
Let $u_0 \in D(\varphi)$ with $u_0 \neq u_{min}$.
Then it holds
\begin{equation}\label{eq:u_*_prop_0}
\varphi(S(t)u_0) < \varphi(u_0)
\quad
\text{if}
\quad
t > 0.
\end{equation}
\end{lem}
\begin{proof}
Let $u_0 \in D(\varphi)$ such that $u_0 \neq u_{min}$.
By Remark\,\ref{rem:contraction}\,(ii), 
we have $\varphi(S(t)u_0) \le \varphi(u_0)$ for any $t \in [0,\infty)$.
Argue by contradiction.
Suppose there exists $T>0$ such that $\varphi(S(T)u_0) = \varphi(u_0)$.
Then it holds that $\varphi(S(t)u_0) = \varphi(u_0)$ for any $t \in [0,T]$.
In \eqref{eq:evi_2} put $v = u_0$, $w = u_0$ and $\tau = t$ \,($t \in (0,T]$), then we have
\begin{align}
  0 = \int_{0}^{t} \left(\varphi\left(S(s)u_0\right) - \varphi(u_0)\right)\,ds
\le -\frac{1}{2}\|S(t)u_0 - u_0\|^2_{L^2}.
\end{align}
Therefore, $S(t)u_0 = u_0$ for any $t \in [0,T]$ 
and $\frac{d}{dt}\left(S(\cdot)u\right)(t) = 0$ for any $t \in (0,T)$.
By Definition\,\ref{dfn:2.1}, there exists $\tau \in (0,T)$ such that
$\frac{d}{dt}\left(S(\cdot)u\right)(\tau) + \partial{\varphi}(S(\tau)u_0) \ni 0$.
Then we have $0 \in \partial{\varphi}(S(\tau)u_0)$.
However, since $\varphi$ is convex, $u_0 = S(\tau)u_0 = u_{min}$.
This is absurd. 
The proof is completed.
\end{proof}

Using the above lemma, we prove Theorem\,\ref{thm:fti}\,(ii).
\begin{proof}[PROOF of Theoerm\,\ref{thm:fti}\,(ii)]
Note that $u_* \neq u_{min}$ by \eqref{assmpt:Lav}.
Then \eqref{eq:u_*_prop_0} implies
 \begin{align}\label{eq:u_*_prop}
\varphi\left(S(t)u_*\right) < \varphi(u_*)
\quad
\text{if}
\quad
t > 0.
\end{align}

In \eqref{eq:evi_int}, we choose $\tau = \frac{\varepsilon}{2},\,v = S\left(\frac{\varepsilon}{2}\right)u_0$ and $w = S\left(\frac{\varepsilon}{2}\right)u_*$.
Then we obtain
\begin{align}
\varphi\left(S(\varepsilon)u_0\right) - \varphi\left(S\left(\frac{\varepsilon}{2}\right)u_*\right)
&= \varphi\left(S\left(\frac{\varepsilon}{2}\right)S\left(\frac{\varepsilon}{2}\right)u_0\right) - \varphi\left(S\left(\frac{\varepsilon}{2}\right)u_*\right)\\
&\le \frac{\|S\left(\frac{\varepsilon}{2}\right)u_0 - S\left(\frac{\varepsilon}{2}\right)u_*\|^2_{L^2}}{\varepsilon}
\le \frac{\|u_0 - u_*\|^2_{L^2}}{\varepsilon}.
\end{align}
by Remark\,\ref{rem:contraction}\,(i).
Therefore, it follows that, for any $t \ge \varepsilon$,
\begin{align}
\varphi\left(S(t)u_0\right)- \varphi\left(u_*\right)
&\le \varphi\left(S(\varepsilon)u_0\right)- \varphi\left(u_*\right) \\
&\le \frac{\|u_0 - u_*\|^2_{L^2}}{\varepsilon} - \left(\varphi\left(u_*\right) -\varphi\left(S\left(\frac{\varepsilon}{2}\right)u_*\right)\right)
\end{align}
since $\varphi\left(S(t)u_0\right)$ is decreasing.
Note that $\varphi\left(u_*\right) -\varphi\left(S\left(\frac{\varepsilon}{2}\right)u_*\right) > 0$ by \eqref{eq:u_*_prop}.
Moreover, since the inequality
$$
\frac{\|u_0 - u_*\|^2_{L^2}}{\varepsilon} - \left(\varphi\left(u_*\right) -\varphi\left(S\left(\frac{\varepsilon}{2}\right)u_*\right)\right) < 0
$$
is equivalent to 
$$
\|u_0 - u_*\|_{L^2} < \sqrt{\varepsilon\left(\varphi\left(u_*\right) -\varphi\left(S\left(\frac{\varepsilon}{2}\right)u_*\right)\right)},
$$
we may choose $\eta = \sqrt{\varepsilon\left(\varphi\left(u_*\right) -\varphi\left(S\left(\frac{\varepsilon}{2}\right)u_*\right)\right)}$.  
The proof is completed.
\end{proof}

\begin{rem}\label{rem:boundedness}
In the above argument, it is sufficient that $u_*$ and $u_{min}$ belong to $D(\widetilde{\varphi})$.
Thus we may assume the boundedness of the boundary data $g$ instead of \eqref{assmpt:2.2.1}.
Indeed, if $g$ belongs to $L^{\infty}(\Omega)$, it follows that $u_*$ and $u_{min}$ also belong to $L^{\infty}(\Omega)$ due to the fact that they are minimizers.
Therefore, $u_*$ and $u_{min}$ belong to $D(\widetilde{\varphi})$.
\end{rem}

\section{Short-time persistence of smooth approximability}

\subsection{Local well-posedness in $C^{-1,\beta}$}
The aim of this section is to prove Theorem\,\ref{thm:ftpr}.
For this purpose, in this subsection we first construct a relatively regular local-in-time solution.
For the construction of such a solution, we use the analytic semigroups generation result in the space $C^{-1,\beta}(\overline{\Omega})$\,\cite{Ve} and the local existence result for abstract fully nonlinear equations \cite{Lu0}.

Throughout this section, we assume \eqref{eq:expo_cond} and \eqref{assmpt:3.1}.
We reduce \eqref{eq:main} to the following homogeneous Dirichlet problem using the given regular function $g$.
\begin{equation}\label{eq:main_hom}
\left\{
\begin{aligned}
\partial_t v = \mathrm{div}f(x,\nabla v + \nabla g)
\quad
&\text{in}\,\Omega \times (0,\infty),\\
v = 0 \quad &\text{on}\, \partial\Omega \times (0,\infty),\\
v(0) = u_0 - g\quad &\text{in}\,\Omega,
\end{aligned}
\right.
\end{equation}
where 
$$
f(x,\xi) = p(|\xi|^2 + \mu)^{\frac{p-2}{2}}\xi + qa(x)(|\xi|^2 + \mu)^{\frac{q-2}{2}}\xi
\quad
\text{for}
\quad
(x,\xi) \in \overline{\Omega} \times \R^n.
$$
Note that for each $x \in \overline{\Omega}$,
we have $f(x,\cdot) \in C^{\infty}(\R^n;\R^n)$ since $\mu > 0$.
Moreover, it holds that for each $1 \le i,j \le n$,
\begin{align}
(D_{\xi}f(x,\xi))_{i,j}
&\coloneq
 p\left\{ \left(|\xi|^2 + \mu \right)^{\frac{p-2}{2}}\delta_{ij} + (p-2)\left(|\xi|^2 + \mu \right)^{\frac{p-4}{2}}\xi_i\xi_j \right\} \\
&+ qa(x)\left\{ \left(|\xi|^2 + \mu \right)^{\frac{q-2}{2}}\delta_{ij} + (q-2)\left(|\xi|^2 + \mu \right)^{\frac{q-4}{2}}\xi_i\xi_j \right\},
\end{align}
where $\delta_{ij}$ is Kronecker's delta.


We now formulate the equation \eqref{eq:main_hom} as an abstract fully nonlinear equation 
by regarding $\mathrm{div}f(x,\nabla v + \nabla g)$ 
as a nonlinear mapping from $C^{1,\beta}(\overline{\Omega})$ to $C^{-1,\beta}(\overline{\Omega})$.
\begin{rem}
One might try to split the divergence term and formulate \eqref{eq:main_hom} as an abstract quasilinear equation as follows;
\begin{align}
\mathrm{div}f(x,\nabla v + \nabla g)
&= \mathrm{div}(p|\nabla v + \nabla g|^2 + \mu)^{\frac{p-2}{2}}\nabla v + qa(x)(|\nabla v + \nabla g|^2 + \mu)^{\frac{q-2}{2}}\nabla v \\
&+ \mathrm{div}(p|\nabla v + \nabla g|^2 + \mu)^{\frac{p-2}{2}}\nabla g + qa(x)(|\nabla v + \nabla g|^2 + \mu)^{\frac{q-2}{2}}\nabla g \\
&\eqcolon -A(v)v + B(v).
\end{align}
In this case, 
the resulting operator $A$ and forcing term $B$ would generally have lower regularity than that required by the local existence theorem for abstract quasilinear equations at least in our setting;
see \cite{Am1,Lu-1,MaW,PrS} for example.
\end{rem}

Fix any $\beta \in (0,1)$ such that $0 < \beta < \min\{\alpha,\gamma\}$.
Set
\begin{align}
  X_0 
  &\coloneq C^{-1,\beta}(\overline{\Omega}) \\
  \coloneq \{&v \in (W^{1,2}_0(\Omega))^*:\,\text{there exist}\,v_0,\ldots,v_n \in  C^{0,\beta}(\overline{\Omega})\\
  &\text{such that}\,v = v_0 + \sum_{i = 1}^{n}D_i v_i\,\text{in the sense of distribution}\},\\
  X_1 &\coloneq \left\{v \in C^{1,\beta}(\overline{\Omega}): v|_{\partial \Omega} = 0 \right\}.
\end{align}
$X_1$ is endowed with the $C^{1,\beta}(\overline{\Omega})$ norm,
and $X_0$ is endowed with the following norm
\begin{align}
  \|v\|_{C^{-1,\beta}(\overline{\Omega})}
  \coloneq
  \inf\left\{\sum_{i = 0}^{n}\|v_i\|_{C^{0,\beta}(\overline{\Omega})}:\,v = v_0 + \sum_{i = 1}^{n}D_i v_i,\,\text{where}\,v_0,\ldots,v_n \in  C^{1,\beta}(\overline{\Omega})\right\}.
\end{align}
We remark that $X_0$ and $X_1$ are Banach spaces.
Moreover, under the natural identification, $X_1$ is the subspace of $X_0$ 
and, moreover, $X_1$ is continuously embedded in $X_0$.
Now we let $F:X_1 \rightarrow X_0$ by
\begin{align}
  F(v) = -\mathrm{div}f(x,\nabla v + \nabla g).
\end{align}
One can see that $F:X_1 \rightarrow X_0$ is Fréchet differentiable.
Moreover, denoting by $dF(v)$ the Fréchet derivative of $F$ at the point $v \in X_1$,
it holds that
\begin{align}
  dF(v)(w) = -\mathrm{div}\left( D_{\xi}f(x,\nabla v + \nabla g) \cdot \nabla w \right)
  \quad
  \text{for any}
  \quad
  v,\,w \in X_1.
\end{align}
Observe that for any $v \in X_1$, 
the each exponent of $D_{\xi}f(x,\nabla v + \nabla g)$ belongs to $C^{0,\beta}(\overline{\Omega})$
and that $D_{\xi}f(x,\nabla v + \nabla g)$ satisfies the ellipticity condition
\begin{align}
D_{\xi}f(x,\nabla v + \nabla g) \eta \cdot \eta \ge \nu|\eta|^2
\quad
\text{for any}
\quad
x \in \overline{\Omega},\,
\eta \in \R^n
\end{align}
with $\nu > 0$, which depends on $p, \mu, \|\nabla v + \nabla g\|_{L^{\infty}(\Omega)}$.
Therefore, the operator $dF(v):X_1 \subset X_0 \rightarrow \mathcal{L}(X_1,X_0)$ generates an analytic semigroup by \cite[Theorem\,5.2]{Ve}.
And the graph norm of $dF(v)$ is equivalent to the norm $\|\cdot\|_{X_1}$.
Furthermore, the Fr\'{e}chet derivative $dF$ belongs to $C^{0,1}_{loc}(X_1,\mathcal{L}(X_1,X_0))$, i.e., for any $v \in X_1$ and any $R>0$, 
there exists $L = L(\mu,v,R) > 0$ such that
\begin{align}
  \|dF(v_1) - dF(v_2)\|_{\mathcal{L}(X_1,X_0)}
  \le L\|v_1 - v_2\|_{X_0}
\end{align}
for any $v_1,v_2 \in \{w \in X_1: \|v - w\| < R\}$.


Now we consider the following abstract equation;
\begin{align}\label{eq:main_abst}
  v'(t) = F(v(t))
  \quad
  0 \le t \le T,
  \quad
  v(0) = v_0.
\end{align}
Then, for a suitable initial value $v_0$, 
\cite[Theorem\,8.11]{Lu0} yields the existence of $T>0$ and $v \in C^1([0,T];X_0) \cap C([0,T];X_1)$
satisfying \eqref{eq:main_abst} in $X_0$.
More precisely, we have the following proposition.
\begin{prop}\label{prop:lwp}
Let $\overline{v} \in X_1$ be such that $F(\overline{v}) \in \overline{X_1}^{X_0}$\,,and let $\theta \in (0,1)$.
Then there exist $T = T(\overline{v}) > 0$ and $r = r(\overline{v}) > 0$ such that
the following assertions hold;
for any $v_0 \in X_1$ with $F(v_0) \in \overline{X_1}^{X_0}$ 
and $\|v - v_0\| \le r$, 
there exists $v \in C^1([0,T];\,X_0) \cap C([0,T];X_1)$ for any $\theta \in (0,1)$, 
which satisfies \eqref{eq:main_abst} in $[0,T]$.
Moreover, $v \in C^{\theta}_{\theta}((0,T];X_1)$ and $v' \in B_{\theta}((0,T];(X_0,X_1)_{\theta,\infty})$.
Here $C^{\theta}_{\theta}$ and $B_{\theta}$ are weighted H\"{o}lder and bounded functional spaces, respectively, see \cite[Chapter\,4]{Lu0} for the precise definition.
In addition, $(X_0.X_1)_{\theta,\infty}$ is an interpolations space, see \cite{Lu0,Ve}.
\end{prop}

We give the characterization of $\overline{X_1}^{X_0}$.
\begin{lem}\label{lem:density}
Assume $\partial\Omega \in C^{1,1}$
and then it holds that
\begin{align}
  \overline{X_1}^{X_0}
  &=
  \{v \in X_0:\,\text{there exist}\,v_0,\ldots,v_n \in  h^{0,\beta}(\overline{\Omega})
  \,\text{such that}\,v = v_0 + \sum_{i = 1}^{n}D_i v_i\}.
\end{align}
Here $h^{0,\beta}(\overline{\Omega})$ is the little H\"{o}lder space\,(see \cite{Lu0,MoMO} 
and the references therein), namely,
\begin{align}
  h^{0,\beta}(\overline{\Omega})
  \coloneq \left\{f \in C^{0,\beta}(\overline{\Omega}): 
  \lim_{\delta \rightarrow 0} \sup_{0 < |x-y| \le \delta,\,x,y \in \overline{\Omega}}\frac{|f(x)-f(y)|}{|x-y|^\beta} = 0 \right\}.
\end{align}
\end{lem}
\begin{proof}[PROOF]
We first prove $h^{-1,\beta}(\overline{\Omega}) \subset \overline{X_1}^{X_0}$.
We divide the proof into three steps.
We first show $C^{\infty}(\overline{\Omega}) \subset \overline{X_1}^{X_0}$.
Let us take any $f \in C^{\infty}(\overline{\Omega})$.
We choose $s \in (n,\infty)$ such that $1 - \frac{n}{s} = \beta$
and take the sequence $\{f_k\}_{k=1}^{\infty} \subset C_c^{\infty}(\Omega)$ such that $f_k \rightarrow f$ in $L^s(\Omega)$.
By \cite[Theorem\,9.15]{GT}, 
there exist $w, w_k \in W^{2,s}(\Omega) \cap W^{1,s}_0(\Omega)$ 
such that $\Delta w = f$ in $\Omega$ and $\Delta w_k = f_k$ in $\Omega$.
Using the Sobolev embedding and \cite[Lemma\,9.17]{GT}, 
we have $\nabla w,\,\nabla w_k \in C^{0,\beta}(\Omega)$ and 
\begin{align}
  \|f_k - f\|_{C^{-1,\beta}(\overline{\Omega})}
  &\le \|\nabla w_k - \nabla w\|_{C^{0,\beta}(\overline{\Omega})} \\
  &\le c\|\nabla w_k - \nabla w\|_{W^{1,s}(\Omega)} \\
  &\le c\|w_k - w\|_{W^{2,s}(\Omega)}
  \le c\|f_k - f\|_{L^s(\Omega)},
\end{align}
where $c$ is a positive constant independent of $k$.
Therefore, letting $k \rightarrow \infty$, 
we obtain $f_k \rightarrow f$ in $X_0$.
Since $f_k \in C_c^{\infty}(\Omega) \subset X_1$, it follows that $f \in \overline{X_1}^{X_0}$.

Moreover, one can verify that $\overline{C^{\infty}(\overline{\Omega})}^{C^{0,\beta}(\overline{\Omega})} 
= h^{0,\beta}(\overline{\Omega})$.
Indeed, we have clearly $C^{\infty}(\overline{\Omega}) \subset h^{0,\beta}(\overline{\Omega})$.
Since $h^{0,\beta}(\overline{\Omega})$ is the closed subspace of $C^{0,\beta}(\overline{\Omega})$ (see the argument of \cite[Proposition\,0.2.1]{Lu0}), 
we obtain 
$\overline{C^{\infty}(\overline{\Omega})}^{C^{0,\beta}(\overline{\Omega})} 
\subset h^{0,\beta}(\overline{\Omega})$. 
In order to obtain the other inclusion, we take any $f \in h^{0,\beta}(\overline{\Omega})$.
By \cite[Corollary 1.8]{MoMO}, there exists a continuous function $\widetilde{f}$ defined in $\R^n$, which is 
an extension of $f$ and satisfies
\begin{align}
 \sup_{x,y \in \R^n,\,x \neq y}\frac{|\widetilde{f}(x)-\widetilde{f}(y)|}{|x-y|^\beta} < \infty
 \quad
 \text{and}
 \quad
 \lim_{\delta \rightarrow 0}\sup_{0 < |x-y| \le \delta,\,x,y\in \R^n} \frac{|\widetilde{f}(x)-\widetilde{f}(y)|}{|x-y|^\beta} = 0.
\end{align}
Then one can see that, letting $\{\rho_k\}$ be a sequence of mollifiers, 
$\rho_k * \widetilde{f} \rightarrow f$ in $C^{0,\beta}(\overline{\Omega})$ as $k \rightarrow \infty$, see the argument of \cite[Proposition\,0.2.1]{Lu0} again.
Since $\rho_k * \widetilde{f} \in C^{\infty}(\overline{\Omega})$, it follows that $f \in \overline{C^{\infty}(\overline{\Omega})}^{C^{0,\beta}(\overline{\Omega})}$.

We now derive the conclusion.
For any $f \in h^{-1,\beta}(\overline{\Omega})$, 
there exist $f_0,\ldots,f_n \in h^{0,\beta}(\overline{\Omega})$ such that
$f = f_0 + \sum_{i=1}^{n}D_if_i$ in the distributional sense.
Since we have $\overline{C^{\infty}(\overline{\Omega})}^{C^{0,\beta}(\overline{\Omega})} 
= h^{0,\beta}(\overline{\Omega})$,
there exist $f_{0,k},\ldots,f_{n,k} \in C^{\infty}(\overline{\Omega})$
such that $f_{i,k} \rightarrow f_{i}$ in $C^{0,\beta}(\overline{\Omega})$ 
for each $i \in \{0,\ldots,n\}$.
Then we put $f_k = f_{0,k} + \sum_{i=1}^{n}D_i f_{i,k} \in C^{\infty}(\overline{\Omega})$. Note that $f_k \in \overline{X_1}^{X_0}$.
And we observe that $f_k \rightarrow f$ in $X_0$. This yields $f \in \overline{X_1}^{X_0}$.

Next we check the other inclusion:\,$\overline{X_1}^{X_0} \subset h^{-1,\beta}(\overline{\Omega})$.
Let us take any $f \in \overline{X_1}^{X_0}$.
Then there exists $\{f_k\}_{k=1}^{\infty} \subset X_1$ such that
$f_k \rightarrow f$ in $X_0$ as $k \rightarrow \infty$.
In particular, we note that $f_k \in C^{-1,\gamma}(\overline{\Omega})$.
Then by \cite[Theorem\,8.34]{GT},
there uniquely exists $w_k \in C^{1,\gamma}(\overline{\Omega})$ with $w_k |_{\partial \Omega} = 0$ such that $\mathrm{div}(\nabla w_k) = f_k$ in the distributional sense.
Similarly, there uniquely exists $w \in X_1$ such that $\mathrm{div}(\nabla w) = f$ in the distributional sense.
Moreover, \cite[Theorem\,8.33]{GT} yields
\begin{align}
  \|\nabla w - \nabla w_k\|_{0,\beta}
  \le \|w - w_k\|_{1,\beta}
  \le c\|f - f_k\|_{-1,\beta},
\end{align}
where the positive constant $c$ is independent of $k$.
Letting $k \rightarrow \infty$, we obtain $\nabla w_k \rightarrow \nabla w$ 
in $ C^{0,\beta}(\overline{\Omega})$.
Since $\nabla w_k \in C^{0,\gamma}(\overline{\Omega}) \subset h^{0,\beta}(\overline{\Omega})$ and 
$h^{0,\beta}(\overline{\Omega})$ is the closed subspace of $C^{0,\beta}(\overline{\Omega})$, 
it follows that $\nabla w \in h^{0,\beta}(\overline{\Omega})$ and $f \in h^{-1,\beta}(\overline{\Omega})$.
The proof is completed.
\end{proof}

For the space $(X_0,X_1)_{\theta,\infty}$, the following result holds (\cite[Theorem\,6.1]{Ve}).
\begin{thm}[\cite{Ve}]\label{thm:interpolation}
Let $\Omega$ be a bounded open set in $\R^N$ with $C^{3,\beta}$ boundary.
Then it holds that, under the natural identifications,
\begin{align}
  (X_0,X_1)_{\theta,\infty}
  =
  \left\{
  \begin{aligned}
  C^{-1,2\theta + \beta}(\overline{\Omega})
  \quad
  &\text{if}
  \quad
  0 < \theta < \frac{1-\beta}{2},\\
  \{v \in C^{0,2\theta + \beta-1}(\overline{\Omega}): v|_{\partial \Omega} = 0 \}
  \quad
  &\text{if}
  \quad
   \frac{1-\beta}{2} < \theta < 1 - \frac{\beta}{2},\\
   \{v \in C^{1,2\theta + \beta-2}(\overline{\Omega}): v|_{\partial \Omega} = 0 \}
    \quad
  &\text{if}
  \quad
   1 - \frac{\beta}{2} < \theta < 1.\\
  \end{aligned}
  \right.
\end{align}
\end{thm}

\subsection{Proof of Theorem\,\ref{thm:ftpr}}
We now prove Theorem\,\ref{thm:ftpr}.
We use the same notation as in Subsection\,3.1
\begin{proof}[PROOF of Theoerm\,\ref{thm:ftpr}]
We take $u_0 \in C^{1,\gamma}(\overline{\Omega})$ such that $u_0|_{\partial \Omega} = g$.
And using Theorem\,\ref{thm:interpolation}, we can choose $\theta \in (0,1)$ such that $(X_0,X_1)_{\theta,\infty}$ is continuously embedded in $L^2(\Omega)$.
Set $v_0 \coloneq u_0 - g$, then we have $v_0 \in X_1$ and $F(v_0) \in C^{-1,\min\{\alpha,\gamma\}}(\overline{\Omega}) \subset \overline{X_1}^{X_0}$ by Lemma \ref{lem:density}.
Therefore, thanks to Proposition\,\ref{prop:lwp}, 
we can find $T = T(v_0) > 0$ and $v \in C^1([0,T];X_0) \cap C([0,T];X_1)$, which satisfies \eqref{eq:main_abst} in $[0,T]$.
Moreover, we have $C^{\theta}_{\theta}((0,T];X_1)$ and $v' \in B_{\theta}((0,T];(X_0,X_1)_{\theta,\infty})$.
In particular, it follows that $v' \in L^{\infty}_{loc}((0,T];L^2(\Omega))$.

Let us define $u \coloneq v + g$.
We observe that $u(t) \in C^{1,\beta}(\overline{\Omega}) \subset X$ for any $t \in [0,T]$ by Remark\,\ref{rem:inclusion}.
Moreover, one can see that $u$ is $L^2$-solution of \eqref{eq:main} with the initial value $u_0$.
Indeed, from the regularity of $v$, it follows that 
$u \in C([0,T];L^2(\Omega)) \cap W^{1,2}_{loc}((0,T);L^2(\Omega))$.
Moreover, we have 
\begin{align}
  \int_{\Omega} \frac{du}{dt}\xi
  +\int_{\Omega}f(x,\nabla u) \cdot \nabla \xi = 0
  \quad
  \text{for any\,}
  \xi \in W^{1,2}_0(\Omega)
  \text{\,and\,a.e.\,}
  t \in (0,T).
\end{align}
If $p < 2$, owing to the spatial regularity of $u$,
the class of admissible test functions can be extended to $W^{1,p}_0(\Omega) \cap L^2(\Omega)$.
Therefore, we have for a.e.\,$t \in (0,T)$, $u(t) \in D(\partial \widetilde{\varphi})$ and $u$ satisfies $\frac{du}{dt}(t) + \partial{\widetilde{\varphi}}(u(t)) \ni 0$ by Lemma\,\ref{lem:charc_subdiff_domain}.
Then we deduce that $u$ is a $L^2$-solution in $[0,T]$.
From the uniqueness of $L^2$-solution (see Remark\,\ref{rem:contraction}\,(iii)), 
it follows $S(t)u_0 = u(t) \in X$ for any \,$t \in [0,T]$.
The proof is completed.

\end{proof}

\section*{Data Availability Statement}
No datasets were generated or analysed during the current study.
\section*{Conflict of interest}
The author has no relevant financial or non-financial interests to disclose.
  


\begin{thebibliography}{9}
  \bibitem{AkM}
  G. Akagi and K. Matsuura, Well-posedness and large-time behaviors of solutions for a parabolic equation involving $p(x)$-Laplacian, Discrete Contin. Dyn. Syst. {\bf 2011}, Dynamical systems, differential equations and applications. 8th AIMS Conference. Suppl. Vol. I, 22--31; MR3012130
  \bibitem{Am1}
  H. Amann, Nonhomogeneous linear and quasilinear elliptic and parabolic boundary value problems, in {\it Function spaces, differential operators and nonlinear analysis (Friedrichroda, 1992)}, 9--126, Teubner-Texte Math., 133, Teubner, Stuttgart; MR1242579
  \bibitem{AmGS}
  L. Ambrosio, N. Gigli and G. Savar\'e, {\it Gradient flows in metric spaces and in the space of probability measures}, second edition, 
Lectures in Mathematics ETH Z\"urich, Birkh\"auser, Basel, 2008; MR2401600
  \bibitem{BaDS}
  A.~K. Balci, L. Diening and M.~D. Surnachev, New examples on Lavrentiev gap using fractals, Calc. Var. Partial Differential Equations {\bf 59} (2020), no.~5, Paper No. 180, 34 pp.; MR4153906
  \bibitem{BaDS_1}
A.~K. Balci, L. Diening and M.~D. Surnachev, Scalar minimizers with maximal singular sets and lack of Meyers property, in {\it Friends in partial differential equations---the Nina N. Uraltseva 90th anniversary volume}, 1--43, EMS Press, Berlin,; MR4967630
 \bibitem{BaCM}
P. Baroni, M. Colombo and G. Mingione, Regularity for general functionals with double phase, Calc. Var. Partial Differential Equations {\bf 57} (2018), no.~2, Paper No. 62, 48 pp.; MR3775180
  \bibitem{Br_0}
  H. Brezis, {\it Op\'erateurs maximaux monotones et semi-groupes de contractions dans les espaces de Hilbert}, North-Holland Mathematics Studies Notas de Matem\'atica, No. 5 No. 50, North-Holland, Amsterdam-London, 1973 American Elsevier Publishing Co., Inc., New York, 1973; MR0348562
  \bibitem{Br}
  H. Brezis, {\it Functional analysis, Sobolev spaces and partial differential equations}, Universitext, Springer, New York, 2011; MR2759829
  \bibitem{CoM1}
  M. Colombo and G. Mingione, Regularity for double phase variational problems, Arch. Ration. Mech. Anal. {\bf 215} (2015), no.~2, 443--496; MR3294408
  \bibitem{CoM2}
  M. Colombo and G. Mingione, Bounded minimisers of double phase variational integrals, Arch. Ration. Mech. Anal. {\bf 218} (2015), no.~1, 219--273; MR3360738
  \bibitem{CoM3}
  M. Colombo and G. Mingione, Calder\'on-Zygmund estimates and non-uniformly elliptic operators, J. Funct. Anal. {\bf 270} (2016), no.~4, 1416--1478; MR3447716
  \bibitem{DeF}
  C. De~Filippis, Gradient bounds for solutions to irregular parabolic equations with $(p, q)$-growth, Calc. Var. Partial Differential Equations {\bf 59} (2020), no.~5, Paper No. 171, 32 pp.; MR4150873
  \bibitem{DeFM0}
  C. De~Filippis and G. Mingione, St. Petersburg Math. J. {\bf 31} (2020), no.~3, 455--477; translated from Algebra i Analiz {\bf 31} (2019), no.~3, 82--115; MR3985927
  \bibitem{EsLM}
  L. Esposito, F. Leonetti and G. Mingione, Sharp regularity for functionals with $(p,q)$ growth, J. Differential Equations {\bf 204} (2004), no.~1, 5--55; MR2076158
  \bibitem{FoMM}
  I. Fonseca, J. Mal\'y{} and G. Mingione, Scalar minimizers with fractal singular sets, Arch. Ration. Mech. Anal. {\bf 172} (2004), no.~2, 295--307; MR2058167
  \bibitem{GT}
  D. Gilbarg and N.~S. Trudinger, {\it Elliptic partial differential equations of second order}, reprint of the 1998 edition, 
Classics in Mathematics, Springer, Berlin, 2001; MR1814364
\bibitem{La}
M.~A. Lavrent\cprime ev, Sur quelques probl\`emes du calcul des variations, Ann. Mat. Pura Appl. {\bf 4} (1927), no.~1, 7--28; MR1553097
  \bibitem{HaH}
  P. Harjulehto and P. H\"ast\"o, {\it Orlicz spaces and generalized Orlicz spaces}, Lecture Notes in Mathematics, 2236, Springer, Cham, 2019; MR3931352
  \bibitem{Ki2}
  W. Kim, Calder\'on-Zygmund type estimate for the singular parabolic double-phase system, J. Math. Anal. Appl. {\bf 551} (2025), no.~1, Paper No. 129593, 33 pp.; MR4897661
  \bibitem{KiKM}
  W. Kim, J. Kinnunen and K. Moring, Gradient higher integrability for degenerate parabolic double-phase systems, Arch. Ration. Mech. Anal. {\bf 247} (2023), no.~5, Paper No. 79, 46 pp.; MR4627284
  \bibitem{KiS}
  W. Kim and L. S\"arki\"o, Gradient higher integrability for singular parabolic double-phase systems, NoDEA Nonlinear Differential Equations Appl. {\bf 31} (2024), no.~3, Paper No. 40, 38 pp.; MR4718687
  \bibitem{KiLV}
  J. Kinnunen, J. Lehrb\"ack and A.~V. V\"ah\"akangas, {\it Maximal function methods for Sobolev spaces}, Mathematical Surveys and Monographs, 257, Amer. Math. Soc., Providence, RI, 2021; MR4306765
  \bibitem{LaSU}
  O.~A. Ladyzhenskaya, V.~A. Solonnikov and N.~N. Ural{\cprime}tseva, {\it Linear and quasilinear equations of parabolic type}, translated from the Russian by S. Smith, 
Translations of Mathematical Monographs, Vol. 23, Amer. Math. Soc., Providence, RI, 1968; MR0241822
\bibitem{Lu-1}
A. Lunardi, Abstract quasilinear parabolic equations, Math. Ann. {\bf 267} (1984), no.~3, 395--415; MR0738260
\bibitem{Lu0}
A. Lunardi, {\it Analytic semigroups and optimal regularity in parabolic problems},
Modern Birkh\"auser Classics, Birkh\"auser/Springer Basel AG, Basel, 1995; MR3012216
\bibitem{MaW}
B.-V. Matioc and C. Walker, On the principle of linearized stability in interpolation spaces for quasilinear evolution equations, Monatsh. Math. {\bf 191} (2020), no.~3, 615--634; MR4064570
\bibitem{MoMO}
K. Mohanta, C. Mudarra and T. Oikari, Traces of vanishing H\"older spaces, J. Geom. Anal. {\bf 35} (2025), no.~1, Paper No. 34, 26 pp.; MR4837231
  \bibitem{PrS}
  J. Pr\"uss and G. Simonett, {\it Moving interfaces and quasilinear parabolic evolution equations}, Monographs in Mathematics, 105, Birkh\"auser/Springer, Cham, 2016; MR3524106
  \bibitem{Si}
  T. Singer, Existence of weak solutions of parabolic systems with $p, q$-growth, Manuscripta Math. {\bf 151} (2016), no.~1-2, 87--112; MR3532237
  \bibitem{Ve}
  V. Vespri, The functional space $C^{-1,\alpha}$ and analytic semigroups, Differential Integral Equations {\bf 1} (1988), no.~4, 473--493; MR0945822
  \bibitem{Zh}
  V.~V. Zhikov, On Lavrentiev's phenomenon, Russian J. Math. Phys. {\bf 3} (1995), no.~2, 249--269; MR1350506
  \end{thebibliography}
\end{document}